\documentclass{amsart}
\usepackage{amsfonts}

\newtheorem{theorem}{Theorem}[section]
\newtheorem{lemma}[theorem]{Lemma}
\theoremstyle{definition}
\newtheorem{definition}[theorem]{Definition}

\newtheorem{proposition}[theorem]{Proposition}
\theoremstyle{remark}
\newtheorem{remark}[theorem]{Remark}

\numberwithin{equation}{section}

\begin{document}
\title[Entropy of Sobolev's classes]{Entropy of Sobolev's classes on Compact
Homogeneous Riemannian Manifolds}
\author{Alexander Kushpel}
\address{Department of Mathematics, University of Leicester, UK}
\email{ak412@le.ac.uk}
\thanks{This research has been supported by the EPSRC Grant EP/H020071/1}
\author{Jeremy Levesley}
\address{Department of Mathematics, University of Leicester, UK}
\email{jl1@le.ac.uk}
\date{April 20th 2015}

\begin{abstract}
We develop a general method to calculate entropy numbers $%
e_{n}(W_{p}^{\gamma },L_{q})$ of standard Sobolev's classes $W_{p}^{\gamma }$
in $L_{q}$ on an arbitrary compact homogeneous Riemannian manifold $\mathbb{M%
}^{d}$. Our method is essentially based on a detailed study of geometric
characteristics of norms induced by subspaces of harmonics on $\mathbb{M}%
^{d} $. The method's possibilities are not confined to the statements proved
but can be applied in studying more general problems such as entropy of
multiplier operators. As an application, we establish sharp orders of $%
e_{n}(W_{p}^{\gamma },L_{q})$ and respective $n$-widths as $n\rightarrow
\infty $ for any $1<p,q<\infty $. In the case $p,q=1,\infty $ sharp in the
power scale estimates are presented.
\end{abstract}

\maketitle

\section{Introduction}

\label{9619.sec1}

Let $(\Omega ,\nu )$ be a measure space and $\{\xi _{k}\}_{k\in \mathbb{N}}$
be a sequence of orthonormal, functions on $\Omega $. Let $X$ be a Banach
space with the norm $\Vert \cdot \Vert _{X}$ and $\{\xi _{k}\}_{k\in \mathbb{%
N}}\subset X$. Clearly, $\Xi _{n}(X):=\mathrm{lin}\{\xi _{1},\cdots ,\xi
_{n}\}\subset X$, $\forall n\in \mathbb{N}$ is a sequence of closed
subspaces of $X$ with the norm induced by $X$. Consider the coordinate
isomorphism $J$ defined as
\[
\begin{array}{ccc}
J:\,\,\mathbb{R}^{n} & \longrightarrow & \Xi _{n}(X) \\
\alpha =(\alpha _{1},\cdots ,\alpha _{n}) & \longmapsto &
\sum_{k=1}^{n}\alpha _{k}\xi _{k}.%
\end{array}%
\end{equation*}%
Hence, the definition
\[
\Vert \alpha \Vert _{J^{-1}\Xi _{n}(X)}=\Vert J\alpha \Vert _{X}
\]%
induces the norm on $\mathbb{R}^{n}$ which appears to be useful in various
applications. Of course, not much can be said regarding such kind of norms
even in lower dimensions. To be able to apply methods of geometry of Banach
spaces to various open problems in different spaces of functions on $\Omega $
we will need to calculate an expectation of the function $\rho _{n}(\alpha
):=\Vert \alpha \Vert _{J^{-1}\Xi _{n}(X)}$ on the unit sphere $\mathbb{S}%
^{n-1}\subset \mathbb{R}^{n}$ with respect to the invariant probabilistic
measure $d\mu _{n}$, i.e., to find the L\'{e}vy mean
\[
M(\Vert \cdot \Vert _{J^{-1}\Xi _{n}(X)})\,=\,\int_{\mathbb{S}^{n-1}}\,\Vert
\alpha \Vert _{J^{-1}\Xi _{n}(X)}\,d\mu _{n}(\alpha ).
\]%
Observe that the sequence of L\'{e}vy means $M(\Vert \cdot \Vert _{J^{-1}\Xi
_{n}(X)})$ contain more information then the sequence of volumes $\mathrm{Vol%
}_{n}\left( B_{J^{-1}\Xi _{n}(X)}\right) $, $n\in \mathbb{N}$, where $%
B_{J^{-1}\Xi _{n}(X)}:=\{\alpha |\,\alpha \in \mathbb{R}^{n},\,\,\Vert
\alpha \Vert _{J^{-1}\Xi _{n}(X)}\leq 1\}$ is the unit ball induced by the
norm $\Vert \cdot \Vert _{J^{-1}\Xi _{n}(X)}$ and therefore is more useful
in various applications.

As a motivating example consider the case $\Omega =\mathbb{M}^{d}$, where $%
\mathbb{M}^{d}$ is a compact homogeneous Riemannian manifold, $\nu $ its
normalized volume element, $\{\xi _{k}\}_{k\in \mathbb{N}}$ is a sequence of
orthonormal harmonics on $\mathbb{M}^{d}$ and $X=L_{p}=L_{p}(\mathbb{M}%
^{d},\nu )$, $p\geq 2$. In general, the sequence $\{\xi _{k}\}_{k\in \mathbb{%
N}}$ is not uniformly bounded on $\mathbb{M}^{d}$. Hence, the method of
estimating of L\'{e}vy means developed in \cite{ku1} - \cite{klw} can not
give sharp order result. Various modifications of this method presented in
\cite{ku3} - \cite{ku5} give an extra $(\log n)^{1/2}$ factor even if $%
p<\infty $. Our general result concentrated in Lemma 3.2 gives sharp order
estimates for the L\'{e}vy means which correspond to the norm induced on $%
\mathbb{R}^{n}$ by the subspace $\oplus _{s=1}^{m}H_{k_{s}}\cap L_{p}$, $%
\mathrm{dim}\,\oplus _{s=1}^{m}H_{k_{s}}:=n$ with an arbitrary index set $%
(k_{1},\cdots ,k_{m})$, where $H_{k_{s}}$ are the eigenspaces of the
Laplace-Beltrami operator for $\mathbb{M}^{d}$ defined by (\ref{lb}). To
show the boundness of the respective L\'{e}vy means as $n\rightarrow \infty $
we impose a technical condition (\ref{df1}) which holds in particular for
any compact homogeneous Riemannian manifold because of the addition formula (%
\ref{addi}) and employ the equality
\[
\int_{\mathbb{R}^{n}}\,h(\alpha )\,d\gamma (\alpha )=\lim_{m\rightarrow
\infty }\,\,\int_{0}^{1}\,h\left( \frac{\delta _{1}^{m}(\theta )}{(2\pi
)^{1/2}},\cdots ,\frac{\delta _{n}^{m}(\theta )}{(2\pi )^{1/2}}\right)
d\theta ,
\]%
where $h:\mathbb{R}^{n}\rightarrow \mathbb{R}$ is a continuous function, $%
h(\alpha _{1},\cdots ,\alpha _{n})\,\mathrm{exp}\left(
-\sum_{k=1}^{n}|\alpha _{k}|\right) \,\rightarrow \,0$ uniformly when $%
\sum_{k=1}^{n}|\alpha _{k}|\,\rightarrow \infty $, $d\gamma (\alpha )=%
\mathrm{exp}\left( -\pi \sum_{k=1}^{n}\alpha _{k}^{2}\right) d\alpha $ is
the Gaussian measure on $\mathbb{R}^{n}$, $\delta _{k}^{m}(\theta
)=m^{-1/2}(r_{(k-1)m}(\theta )+\cdots +r_{km})$, $1\leq k\leq n$ and $%
r_{s}(\theta )=\mathrm{sign}\sin (2^{s}\pi \theta )$, $s\in \mathbb{N}\cup
\{0\}$, $\theta \in \lbrack 0,1]$ is the sequence of Rademacher functions
\cite{kwap}, \cite{kuto1}. To extend our estimates to the case $p=\infty $
we apply Lemma \ref{lemimbed} which gives a useful inequality between $1\leq
p,q\leq \infty $ norms of polynomials on $\mathbb{M}^{d}$ with an arbitrary
spectrum. It seems that the factor $(\log n)^{1/2}$ obtained in Lemma \ref%
{levymean} is essential because of the lower bound for the L\'{e}vy means
found in \cite{kashin} in the case of trigonometric system. This fact
explains a logarithmic slot in our estimates of entropy numbers presented in
Theorem \ref{entropy}. Section 3 deals with estimates of entropy numbers and
$n$-widths. Theorem \ref{lower1} establishes general lower bounds for
entropy numbers in terms of L\'{e}vy means and is of independent interest.
We derive lower bounds for the entropy numbers of Sobolev's classes (\ref%
{sob-ent}) using Theorem \ref{lower1} and estimates of L\'{e}vy means given
by Lemma \ref{levymean}. At this point we apply Lemma \ref{ratios} to get
the dependence between eigenvalues and dimensions of eigenspaces of the
Laplace-Beltrami operator. The proof of Lemma \ref{ratios} is based on
Weyl's formula (see \cite{xxx})
\begin{equation}
\lim_{a\rightarrow \infty }a^{-d/2}n(a)=(2\pi ^{1/2})^{-d}\Gamma \left( 1+%
\frac{d}{2}\right) V(\mathbb{M}^{d}),  \label{weyl11}
\end{equation}%
where $V(\mathbb{M}^{d})$ is the volume of $\mathbb{M}^{d}$ and $n(a)$ is
the number of eigenvalues (each counted with its multiplicity) smaller than $%
a$. To get upper bounds for entropy numbers contained in Theorem \ref%
{entropy} we apply estimates of L\'{e}vy means established in Lemma \ref%
{levymean} and make use of the Pajor-Tomczak-Jaegermann inequality \cite{pt}
which states in our notations that for any $\lambda \in (0,1)$ there exists
a subspace $X_{s}\subset J^{-1}\Xi _{n}(X)$, $\mathrm{dim}\,X_{s}=s>\lambda
n $ and a universal constant $C>0$ such that
\begin{equation}
|\alpha |\leq C\frac{M(\Vert \cdot \Vert _{J^{-1}\Xi _{n}(X)}^{o})}{%
(1-\lambda )^{1/2}}\Vert \alpha \Vert _{J^{-1}\Xi _{n}(X)},\,\,\,\forall
\alpha \in X_{s},  \label{pt1}
\end{equation}%
where $|\cdot |$ is the Euclidean norm on $\mathbb{R}^{n}$ and $\Vert \cdot
\Vert _{J^{-1}\Xi _{n}(X)}^{o}$ is the dual norm with respect to $\Vert
\cdot \Vert _{J^{-1}\Xi _{n}(X)}$. Remark that (\ref{pt1}) is essentially
based on a technical result due to Gluskin \cite{gluskin}. Hawever, for our
applications is sufficient to apply a less sharp result established by
Bourgain and Milman \cite{burm} which is based on averaging arguments and
isoperimetric inequality.

The paper ends with estimates of different $n$-widths and their applications
in calculation of entropy which extend previous results \cite{bklt111}, \cite%
{jfan}, \cite{ku5}.

In this article there are several universal constants which enter into the
estimates. These positive constants are mostly denoted by $C,C_{1},...$. We
will only distinguish between the different constants where confusion is
likely to arise, but we have not attempted to obtain good estimates for
them. For ease of notation we will write $a_{n} \ll b_{n}$ for two
sequences, if $a_{n} \leq Cb_{n}$, $\forall n \in \mathbb{N}$ and $a_{n}
\asymp b_{n}$, if $C_{1}b_{n} \leq a_{n} \leq C_{2}b_{n}$, $\forall n \in
\mathbb{N}$ and some constants $C$, $C_{1}$ and $C_{2}$. Also, we shall put $%
(a)_{+}:=\max\{a,0\}$.

Though the main purpose of this paper is to present new results, we have
tried to make the text selfcontained by presenting well known definitions
and elementary properties of entropy numbers and $n$-widths.

Let $X$ and $Y$ be Banach spaces with the closed unit balls $B_{X}$ and $%
B_{Y}$ respectively. Let $v:X\rightarrow Y$ be a compact operator. Then the $%
n$th entropy number $e_{n}(v)=e_{n}(v:X\rightarrow Y)$ is the infimum of all
positive $\epsilon $ such that there exist $y_{1},\cdots ,y_{2^{n-1}}$ in $Y$
such that
\[
v(B_{X})\subset \bigcup_{k=1}^{2^{n-1}}(y_{k}+\epsilon B_{Y}).
\]%
Similarly, for a compact set $A\subset Y$ we define the entropy number $%
e_{n}(A,Y)$ as the infimum of all positive $\epsilon $ such that there exist
$\{y_{k}\}_{k=1}^{2^{n-1}}\subset Y$ such that $A\subset
\bigcup_{k=1}^{2^{n-1}}(y_{k}+\epsilon B_{Y}).$ Suppose that $A$ is a
convex, compact, centrally symmetric subset of a Banach space $X$ with unit
ball $B_{X}$. The Kolmogorov $n$-width of $A$ in $X$ is defined by $%
d_{n}(A,X):=d_{n}(A,B_{X}):=\inf_{X_{n}}\,\sup_{f\in A}\,\inf_{g\in
X_{n}}\Vert f-g\Vert _{X},$ where $X_{n}$ runs over all subspaces of $X$ of
dimension $n$. The Gel'fand $n$-width of $A$ in $X$ is defined by $%
d^{n}(A,X):=d^{n}(A,B_{X}):=\inf_{L^{n}}\,\sup_{x\in L^{n}\cap A}\Vert
x\Vert _{X}$, where $L^{n}$ runs over all subspaces of $X$ of codimension $n$%
. The Bernstein $n$-width of $A$ in $X$ is defined by $%
b_{n}(A,X):=b_{n}(A,B_{X}):=\sup_{X_{n+1}}\,\sup \{\epsilon >0:\,\epsilon
B\cap X_{n+1}\subset A\},$ where $X_{n+1}$ is any $(n+1)$-dimensional
subspace of $X$. For a compact operator $v:X\rightarrow Y$ we define
Kolmogorov's numbers
\[
d_{n}(v)=d_{n}(v:X\rightarrow Y)=\inf_{L\subset Y,\,\mathrm{dim}\,L\leq
n}\,\,\sup_{x\in B_{X}}\,\inf_{y\in L}\Vert vx-y\Vert _{Y}
\]%
and Gelfand's numbers
\[
d^{n}(v)=d^{n}(v:X\rightarrow Y)=\inf \{\Vert v|L\Vert \,|L\subset X,\,%
\mathrm{codim}L\leq n\}.
\]

\begin{proposition}
\label{simple} This proposition records some simple properties of $n$-widths
and entropy numbers.

\begin{enumerate}
\item If $X \subset Y$, then $d_n(A,Y) \le d_n(A,X)$.

\item Let $n=i+j$ and $A=A_{1} + A_{2}$. Then $d_n(A,X) \le
d_i(A_{1},X)+d_j(A_{1},X)$.

\item Kolmogorov and Gel'fand $n$-widths are dual. Let $X$ and $Y$ be Banach
spaces, $v \in \mathcal{L}(X, Y)$. If $X$ is reflexive and $v(X)$ is dense
in $Y$, then $d_{n}(v) = d^{n}(v^{\ast})$ (see, e.g., \cite{lgm}, p.408).

\item Later we will wish to restrict estimation of entropy numbers over
infinite-dimensional sets to finite-dimensional sets. In order to do this
let $i$ be any linear isometry, $i:Y\rightarrow \tilde{Y}$ (here we will
think of $Y$ as finite dimensional and $i$ as the imbedding into the
infinite dimensional space). Then (\cite[Proposition 5.1]{pisier}) $%
2^{-1}e_{n}(v)\leq e_{n}(i\circ v)\leq e_{n}(v),\,\,\,\forall n\in \mathbb{N}%
.$
\end{enumerate}
\end{proposition}

\section{Elements of Harmonic Analysis on Compact Riemannian Homogeneous
Manifolds}

\label{harmonic analysis}

\begin{definition}
Let $(\Omega ,\nu )$ be a measure space for some compact set $\Omega \in
\mathbb{R}^{s}$, $s\in \mathbb{N}$. Let $\Xi =\{\xi _{k}\}_{k\in \mathbb{N}}$
be a set of orthonormal functions in $L_{2}(\Omega ,\nu )$. Suppose that
there exists a sequence $\kappa =\{k_{j}\}_{j\in \mathbb{N}}$, $k_{1}=1$,
such that for any $j\in \mathbb{N}$ and some $C>0$
\begin{equation}
\sum_{k=k_{j}}^{k_{j+1}-1}|\xi _{k}(x)|^{2}\leq Cd_{j}  \label{df1}
\end{equation}%
a.e. on $\Omega $, where $d_{j}:=k_{j+1}-k_{j}$. Then we say that $(\Omega
,\nu ,\Xi ,\kappa )\in {\mathcal{K}}$.
\end{definition}

Consider the set of $p$-integrable functions on $(\Omega ,\nu )$, $%
L_{p}=L_{p}(\Omega ,\nu )$. It follows from (\ref{df1}) that the functions $%
\xi _{k}$ are a.e. bounded for every $n\in \mathbb{N}$. Hence, for an
arbitrary $\phi \in L_{p}$, $1\leq p\leq \infty $ it is possible to
construct the Fourier coefficients
\[
c_{k}(\phi )=\int_{\Omega }\phi \overline{\xi _{k}}d\nu ,\quad k\in {\mathbf{%
N}},
\]%
and consider the formal Fourier series
\[
\phi \sim \sum_{l\in \mathbb{N}}\sum_{k=k_{l}}^{k_{l+1}-1}c_{k}(\phi )\xi
_{k}.
\]%
Let $U_{p}:=\{\phi |\,\Vert \phi \Vert _{p}\leq 1\}$ be the unit ball in $%
L_{p}$, and $\Lambda =\{\lambda _{l}\}_{l\in \mathbb{N}}$ be a fixed
sequence of complex numbers. We shall say that the multiplier operator $%
\Lambda $ is of type $(\kappa ,p,q)$ with the norm $\Vert \Lambda \Vert
_{p,q}^{\kappa }:=\sup_{\phi \in U_{p}}\Vert \Lambda \phi \Vert _{q}$, , if
for any $\phi \in L_{p}$ there is such $f\in L_{q}$ that
\[
f\sim \sum_{l\in \mathbb{N}}\lambda _{l}\sum_{k=k_{l}}^{k_{l+1}-1}c_{k}(\phi
)\xi _{k}.
\]%
Let us present here several important examples of measure spaces \newline
$(\Omega ,\nu ,\Xi ,\kappa )\in {\mathcal{K}}$. Consider a compact,
connected, $d$-dimensional $C^{\infty }$ Riemannian manifold $\mathbb{M}^{d}$
with $C^{\infty }$ metric. Let $g$ its metric tensor, $\nu $ its normalized
volume element and $\Delta $ its Laplace-Beltrami operator. In local
coordinates $x_{l}$, $1\leq l\leq d$,
\begin{equation}
\Delta =-(\overline{g})^{-1/2}\sum_{k}\frac{\partial }{\partial x_{k}}\left(
\sum_{j}g^{jk}(\overline{g})^{1/2}\frac{\partial }{\partial x_{j}}\right) ,
\label{lb}
\end{equation}%
where $g_{jk}:=g(\partial /x_{j},\partial /x_{k})$, $\overline{g}:=|\mathrm{%
det}(g_{jk})|$, and $(g^{jk}):=(g_{jk})^{-1}$. It is well-known that $\Delta
$ is an elliptic, self adjoint, invariant under isometry, second order
operator. The eigenvalues $\theta _{k}$, $k\geq 0$, of $\Delta $ are
discrete, nonnegative and form an increasing sequence $0\leq \theta _{0}\leq
\theta _{1}\leq \cdots \leq \theta _{n}\leq \cdots $ with $+\infty $ the
only accumulation point. The corresponding eigenspaces $H_{k}$, $k\geq 0$
are finite-dimensional, $d_{k}:=\mathrm{dim}\,(H_{k})$, orthogonal and $%
L_{2}=L_{2}(\mathbb{M}^{d},\nu )=\oplus _{k=0}^{\infty }H_{k}$. Let us fix
an orthonormal basis $\{Y_{m}^{k}\}_{m=1}^{d_{k}}$ of $H_{k}$. For an
arbitrary $\phi \in L_{p}$, $1\leq p\leq \infty $ with the formal Fourier
series
\[
\phi \sim c_{0}+\sum_{k\in \mathbb{N}}\sum_{m=1}^{d_{k}}c_{k,m}(\phi
)Y_{m}^{k},\,\,\,c_{k,m}(\phi )=\int_{\mathbb{M}^{d}}\phi \overline{Y_{m}^{k}%
}d\nu ,
\]%
the $\gamma $-th fractional integral $I_{\gamma }\phi :=\phi _{\gamma }$, $%
\gamma >0$, is defined as
\[
\phi _{\gamma }\sim c+\sum_{k\in \mathbb{N}}\theta _{k}^{-\gamma
/2}\sum_{m=1}^{d_{k}}c_{k,m}(\phi )Y_{m}^{k},\,\,\,c\in \mathbb{R}.\,\,\,\
\]%
The function $D_{\gamma }\phi :=\phi ^{(\gamma )}\in L_{p}$, $1\leq p\leq
\infty $ is called the $\gamma $-th fractional derivative of $\phi $ if
\[
\phi ^{(\gamma )}\sim \sum_{k\in \mathbb{N}}\theta _{k}^{\gamma
/2}\sum_{m=1}^{d_{k}}c_{k,m}(\phi )Y_{m}^{k}.
\]%
The Sobolev classes $W_{p}^{\gamma }$ are defined as sets of functions with
formal Fourier expansions (\ref{sobolev}) where $\Vert \phi \Vert _{p}\leq 1$
and $\int_{\mathbb{M}^{d}}\phi d\nu =0$. In this article we assume $%
c_{0},c=0 $ to guarantee compactness of the set $W_{p}^{\gamma }$ in $L_{q}$.

We recall that a Riemannian manifold $\mathbb{M}^{d}$ is called homogeneous
if its group of isometries ${\mathcal{G}}$ acts transitively on it, i.e. for
every $x,y\in \mathbb{M}^{d}$, there is a $g\in {\mathcal{G}}$ such that $%
gx=y$. For a compact homogeneous Riemannian manifold $\mathbb{M}^{d}$ the
following addition formula is known \cite{gine}
\begin{equation}
\sum_{k=1}^{d_{k}}|Y_{m}^{k}(x)|^{2}=d_{k},\,\,\,\forall x\in \mathbb{M}^{d},
\label{addi}
\end{equation}%
where $\{Y_{m}^{k}\}_{m=1}^{d_{k}}$ is an arbitrary orthonormal basis of $%
H_{k}$, $k\geq 0$. Hence, any such manifold possesses the property ${%
\mathcal{K}}$, and these include real and complex Grassmannians, the $n$%
-torus, the Stiefel manifold, two point homogeneous spaces (spheres, the
real, complex and quaternionic projective spaces and the Cayley elliptic
plane), and the complex sphere.

\begin{lemma}
\label{ratios} Let $\mathbb{M}^{d}$ be a compact, connected, homogeneous
Riemannian manifold, $\{\theta_{k}\}_{k \in \mathbb{N} \cup \{0\}}$ be the
sequence of eigenvalues and $\{H_{k}\}_{k \in \mathbb{N} \cup \{0\}}$ be the
corresponding sequence of eigenspaces of the Laplace-Beltrami operator $%
\Delta$ on $\mathbb{M}^{d}$. Put $\mathcal{T}_{N} = \oplus_{k=0}^{N} H_{k}$
and $\tau_N=\mathrm{dim} \, \mathcal{T}_{N}$.
\begin{equation}  \label{quoteig}
\lim_{N \rightarrow \infty} \frac{\theta_{N+1}}{\theta_{N}} =1
\end{equation}
and
\begin{equation}  \label{quotpoly}
\lim_{N \rightarrow \infty} \frac{\tau_{N+1}}{\tau_N} =1.
\end{equation}
\end{lemma}

\begin{proof}
Applying Weyl's formula (\ref{weyl11}) for $a=\theta_{N}$ we get
\begin{equation}  \label{weyl}
\lim_{N \rightarrow \infty} \theta_N^{-d/2} n(\theta_N) = (2
\pi^{1/2})^{-d}\Gamma\left(1 + \frac{d}{2}\right)V(\mathbb{M}^{d}),
\end{equation}
and it follows that
\[
\lim_{N \rightarrow \infty} \frac{\theta_{N+1}^{-d/2}
n(\theta_{N+1})-\theta_N^{-d/2} n(\theta_N)}{\theta_N^{-d/2} n(\theta_N)}
\rightarrow 0, \quad N \rightarrow \infty.
\]
Now, $n(\theta_N) = \tau_N$, so that
\begin{eqnarray}
\lefteqn{\frac{\theta_{N+1}^{-d/2} \tau_{N+1} -\theta_N^{-d/2} \tau_N}{%
\theta_N^{-d/2} \tau_N } }  \notag \\
& = & \frac{\theta_{N+1}^{-d/2} (\tau_N + \mathrm{dim} \, H_{N+1})
-\theta_N^{-d/2} \tau_{N}}{\theta_N^{-d/2} \tau_N }  \notag \\
& = & {\frac{\theta_{N+1}^{-d/2} -\theta_N^{-d/2} }{\theta_N^{-d/2} }} + {%
\frac{\theta_{N+1}^{-d/2} }{\theta_N^{-d/2} }} {\frac{\mathrm{dim} \,
H_{N+1} }{\tau_N }} \rightarrow 0, \quad N \rightarrow \infty.
\label{quotients}
\end{eqnarray}
Since both quotients in the last equation are positive we have
\[
{\frac{\theta_{N+1}^{-d/2} -\theta_N^{-d/2} }{\theta_N^{-d/2} }} \rightarrow
0, \quad N \rightarrow \infty,
\]
which gives us (\ref{quoteig}). Equation (\ref{quotpoly}) follows since
\[
\lim_{N \rightarrow \infty} {\frac{ \tau_{N+1} }{\tau_N }} = \lim_{N
\rightarrow \infty} {\frac{ \tau_N + \mathrm{dim} \, H_{N+1} }{\tau_N }} = 1,
\]
using the second quotient in (\ref{quotients}), and (\ref{quoteig}).
\end{proof}

\section{Estimates of Entropy and $n$-Widths}

\label{entropy widths}

In this section we give several estimates of entropy and $n$-widths which
are order sharp in many important cases. Fix a measure space $(\Omega, \nu)$%
, an orthonormal system $\Xi$ and a sequence $\{k_{j}\}_{j \in \mathbb{N}}$
such that $(\Omega, \nu, \Xi, \kappa) \in {\mathcal{K}}$. Let
\[
\Xi^{j} := \mathrm{span} \, \{\xi_{k}\}_{k=k_{j}}^{k_{j+1}-1},\,\,\,
\Omega_{m} := \{j_{1}, \cdots , j_{m}\},\,\,\,\, \Xi(\Omega_{m}) := \mathrm{%
span} \, \{\Xi^{j_{s}}\}_{s=1}^{m}.
\]
Put $l_0:=0$, $l_k := \sum_{s=1}^{k} d_{j_{s}}$, $k=1,\cdots,m$, and $n :=
l_m = \mathrm{dim} \, (\Xi(\Omega_{m}))$.

Unfortunately, we need to introduce a re-enumeration of the functions $\xi_k$%
, since we are selecting separated blocks of them. Let us write
\[
(\Xi(\Omega_{m}))= \mathrm{span} \, \{ \eta_i : i=1,\cdots,n \},
\]
with the $\eta_i$ organized so that $\Xi^{j_s} = \mathrm{span} \, \{
\eta_{i} : l_{s-1}+1 \le i \le l_s \}$. Consider the coordinate isomorphism
\[
J:\,\mathbb{R}^{n} \rightarrow \Xi(\Omega_{m})
\]
that assigns to $\alpha =(\alpha_{1} \cdots , \alpha_{n}) \in \mathbb{R}^{n}$
the function $J\alpha = \xi^{\alpha} = \sum_{l=1}^{n} \alpha_{l} \eta_l \in
\Xi(\Omega_{m})$. Let $X$ and $Y$ be given Banach spaces such that $%
\Xi(\Omega_{m}) \subset X$ and $\Xi(\Omega_{m}) \subset Y$ for any $%
\Omega_{m} \subset \mathbb{N}$. Put $X_{n} = \Xi(\Omega_{m}) \cap X$ and $%
Y_{n} = \Xi(\Omega_{m}) \cap Y$. Let $\lambda_i \in \mathbb{R}$, $%
i=1,\cdots,m$, and
\[
\Lambda_{n} = \mathrm{diag} \{\lambda_{1} I_{d_{j_1}}, \cdots, \lambda_{m}
I_{d_{j_m}}\},
\]
where $I_s$ is the identity matrix of dimension $s$. Now, if $\Lambda_n$ is
invertible then $J \Lambda_n J^{-1}:\,\Xi(\Omega_{m}) \rightarrow
\Xi(\Omega_{m})$ is an invertible operator which essentially multiplies each
block $\Xi^{j_s}$ by $\lambda_s$, $s=1,\cdots,m$. Since it should not cause
any confusion we will refer to this operator also as $\Lambda_n$.

In what follows, a $*$ will be used to denote norms and balls in Euclidean
space, and lack of a $*$ will indicate the same quantities in function
spaces. Let us define the norms
\[
\|\alpha\|^*_{X_{n}} = \|\xi^\alpha\|_{X_{n}} : = \| \xi^{\alpha}\|_{X}.
\]
Put $B^{*}_{X_{n}}:=\{\alpha \in \mathbb{R}^{n}, \; \|\alpha\|^*_{X_{n}}
\leq 1\}, $ and $B_{X_{n}} := J B^*_{X_{n}}$.

\begin{lemma}
\label{lemimbed} For any $\Omega_{m}$ and any $\xi \in \Xi(\Omega_{m})$, $m
\in \mathbb{N}$ we have
\[
\|\xi\|_{p} \leq Cn^{(1/p-1/q)_{+}} \|\xi\|_{q},\,\,\,
\]
where $1 \leq p,q \leq \infty$ and $n := \mathrm{dim}\, \Xi(\Omega_{m})$.
\end{lemma}

\begin{proof}
Let
\[
K_{n}(x,y):=\sum_{i=1}^{n}\eta _{i}(x)\overline{\eta _{i}(y)}.
\]%
be the reproducing kernel for $\Xi (\Omega _{m})$. Clearly,
\[
K_{n}(x,y)=\int_{\mathbb{M}^{d}}K_{n}(x,z)K_{n}(z,y)d\nu (z),
\]%
and $K_{n}(x,y)=\overline{K_{n}(y,x)}$. Hence, using the Cauchy-Schwarz
inequality,
\[
\Vert K_{n}(\cdot ,\cdot )\Vert _{\infty }\leq \Vert K_{n}(y,\cdot )\Vert
_{2}\Vert K_{n}(x,\cdot )\Vert _{2}
\]%
for any $x,y\in \mathbb{M}^{d}$. Since $(\Omega ,\nu ,\Xi ,\kappa )\in
\mathcal{K}$, from (\ref{df1}), we have $\Vert K_{n}(x,\cdot )\Vert _{2}\leq
Cn^{1/2}$. Therefore,
\begin{equation}  \label{infty}
\Vert K_{n}(\cdot ,\cdot )\Vert _{\infty }\leq Cn.
\end{equation}%
Let $\xi \in \Xi (\Omega _{m})$. Then applying H\"{o}lder inequality and (%
\ref{infty}) we get
\[
\Vert \xi \Vert _{\infty }\leq \Vert K_{n}(\cdot ,\cdot )\Vert _{\infty
}\Vert \xi \Vert _{1},
\]%
and hence
\[
\Vert I\Vert _{L_{1}(\mathbb{M}^{d})\cap \Xi (\Omega _{m})\rightarrow
L_{\infty }(\mathbb{M}^{d})\cap \Xi (\Omega _{m})}\leq Cn,
\]%
where $I$ is the embedding operator. Trivially, $\Vert I\Vert _{L_{p}(%
\mathbb{M}^{d})\cap \Xi (\Omega _{m})\rightarrow L_{p}(\mathbb{M}^{d})\cap
\Xi (\Omega _{m})}=1$, $1\leq p\leq \infty $. Hence, using the Riesz-Thorin
interpolation Theorem and embedding arguments, for any $\xi \in \Xi (\Omega
_{m})$, we obtain
\[
\Vert \xi \Vert _{p}\leq Cn^{(1/p-1/q)_{+}}\Vert \xi \Vert _{q},\,\,\,1\leq
p,q\leq \infty .
\]
\end{proof}

Let us fix a norm $\Vert \cdot \Vert $ on $\mathbb{R}^{n}$ and let $E=(%
\mathbb{R}^{n},\Vert \cdot \Vert )$ be a Banach space with the unit ball $%
B_{E}$. The dual space $E^{o}=(\mathbb{R}^{n},\Vert \cdot \Vert ^{o})$ is
endowed with the norm $\Vert \xi \Vert ^{o}=\sup_{\sigma \in B_{E}}|\langle
\xi ,\sigma \rangle |$ and has the unit ball $B_{E^{o}}$. In these notations
the L\'{e}vy mean $M_{B_{E}}$ is
\[
M_{B_{E}}=\int_{\mathbb{S}^{n-1}}\Vert \xi \Vert d\mu ,
\]%
where $d\mu $ denotes the normalized invariant measure on $\mathbb{S}^{n-1}$%
, the unit sphere in $\mathbb{R}^{n}$. We are interested in the case where $%
\Vert \cdot \Vert =\Vert \alpha \Vert _{I_{n},L_{p}}^{\ast }$. In the case $%
\Omega _{m}=\{1,\cdots ,m\}$ the estimates of the associated L\'{e}vy means
were obtained in \cite{kuto1}. Using Lemma~\ref{lemimbed} we can
straightforwardly generalize this result to be valid for an arbitrary index
set $\Omega _{m}$. This we state as

\begin{lemma}
\label{levymean} Let $X_{n}=L_{p}\cap \Xi (\Omega _{m})$, and $n:=\mathrm{dim%
}\,\Xi (\Omega _{m})$. Then the L\'{e}vy mean

$\ \ M_{B_{X_{n}}^{\ast }}\leq C\left\{
\begin{array}{cc}
p^{1/2}, & p<\infty , \\
(\log n)^{1/2}, & p=\infty .%
\end{array}%
\right. $
\end{lemma}

We can now give lower bounds for entropy in terms of L\'{e}vy means. In the
following the reader should be identifying the spaces $X$ and $Y$ with $%
L_{p} $ and $L_{q}$ respectively for some $1\leq p,q\leq \infty $. However,
we wished to state the result in greater generality, and then apply the
previous result to extract particular results in these cases.

\begin{theorem}
\label{lower1} Viewing $\Xi(\Omega_{m})$ as a subspace $X_n \subset X$ and $%
Y_n \subset Y$, we have
\[
e_{k}(\Lambda U_{X}, Y) \geq 2^{-1-k/n} \frac{|\mathrm{det}\,
\Lambda_{n}|^{1/n}}{M_{B^{*}_{X_{n}}} M_{(B^{*}_{Y_{n}})^{o}}},\,\,
\]
where $k,n \in \mathbb{N}$ are arbitrary.
\end{theorem}

\begin{proof}
First, we use Proposition \ref{simple} (d) to obtain the estimate
\begin{equation}  \label{iii}
e_{k}(\Lambda U_{X}, Y) \geq 2^{-1} e_{k}(\Lambda U_{X} \cap \Xi_{n}, Y \cap
\Xi_{n}) = 2^{-1} e_k (\Lambda_{n} ( B^{*}_{X_{n}}) , B^{*}_{Y_{n}} ),
\end{equation}
using the appropriate norms in $X_n$ and $Y_n$. Let $\vartheta_{1}, \cdots ,
\vartheta_{N(\epsilon)}$ be a minimal $\epsilon$-net for $\Lambda_{n} (
B^{*}_{ X_{n}}) $ in $\left(B^{*}_{Y_{n}}, \,\mathbb{R}^{n}\right)$. Then,
\[
\Lambda_{n} B^{*}_{X_{n}} \subset \bigcup_{k=1}^{N(\epsilon)} \left(\epsilon
B^{*}_{Y_{n}} + \vartheta_{k} \right).
\]
By comparing volumes we get
\begin{eqnarray*}
\mathrm{Vol}_{n}\left(\Lambda_{n} B^{*}_{ X_{n}}\right) & = & | \, \mathrm{%
det} \, \Lambda_{n}| \mathrm{Vol}_{n}\left( B^{*}_{X_{n}}\right) \\
& \leq & \epsilon^{n} N(\epsilon) \mathrm{Vol}_{n} \left(B^{*}_{
Y_{n}}\right).
\end{eqnarray*}
If we put $N(\epsilon) = 2^{k-1}$, then from the last inequality and the
definition of entropy numbers we obtain
\begin{equation}  \label{epsilon}
\epsilon = e_{k}\left(\Lambda_{n} B^{*}_{X_{n}},B^{*}_{Y_{n}}\right) \geq
2^{-k/n} | \, \mathrm{det} \, \Lambda_{n}|^{1/n} \left( \frac{\mathrm{Vol}%
_{n} \left(B^{*}_{X_{n}}\right)}{\mathrm{Vol}_{n} \left(B^{*}_{Y_{n}}\right)}%
\right)^{1/n}.
\end{equation}
Let $B^{*}_2$ be the unit Euclidian ball in $\mathbb{R}^{n}$, $V \subset
\mathbb{R}^{n}$ be a convex symmetric body and $V^{o}$ its dual. From
Uryson's inequality \cite[P. 6]{pisier}
\[
\left(\frac{\mathrm{Vol}_{n} V}{\mathrm{Vol}_{n} B_{2}^{*}}\right)^{1/n}
\leq M_{V^{o}}
\]
it follows that
\[
\left(\mathrm{Vol}_{n} \left(B^{*}_{Y_{n}}\right)\right)^{1/n} \leq
M_{(B^{*}_{Y_{n}})^{o}}\left(\mathrm{Vol}_{n}
\left(B^{*}_{2}\right)\right)^{1/n},
\]
so that
\begin{eqnarray}  \label{bnd1}
\left( \frac{\mathrm{Vol}_{n} \left(B^{*}_{X_{n}}\right)}{\mathrm{Vol}_{n}
\left(B^{*}_{Y_{n}}\right)}\right)^{1/n} & \ge & \frac{\left(\mathrm{Vol}%
_{n} \left(B^{*}_{X_{n}}\right) \right)^{1/n}}{M_{(B^{*}_{Y_{n}})^{o}}\left(%
\mathrm{Vol}_{n} \left(B^{*}_{2} \right )\right)^{1/n}}.
\end{eqnarray}
Also, a direct calculation shows that
\begin{equation}  \label{111}
\left(\frac{\mathrm{Vol}_{n} \left(B^{*}_{X_{n}}\right)} {\mathrm{Vol}_{n}
\left(B^{*}_{2}\right)}\right)^{1/n}\,= \,\left(\int_{\mathbb{S}%
^{n-1}}\,\|\alpha\|^{-n} d\mu(\alpha)\right)^{1/n}\, \geq
M_{B^{*}_{X_{n}}}^{-1}.
\end{equation}
Combining (\ref{iii}), (\ref{bnd1}) and (\ref{111}) we have
\[
\left( \frac{\mathrm{Vol}_{n} \left(B^{*}_{X_{n}}\right)}{\mathrm{Vol}_{n}
\left(B^{*}_{Y_{n}}\right)}\right)^{1/n} \geq \frac{1}{M_{B^{*}_{X_n}}
M_{(B^{*}_{Y_{n}})^{o}}}.
\]
and substitution into (\ref{epsilon}) completes the proof.
\end{proof}

\begin{remark}
Assume that $|\lambda_{1}| \geq \cdots \geq |\lambda_{n}|$. Then $| \,
\mathrm{det} \, \Lambda_{n}|^{1/n} \geq |\lambda_{n}|$, and for $k=n$ we
have
\begin{equation}  \label{ent}
e_{n}(\Lambda U_{X}, Y) \geq \frac{|\lambda_{n}|}{ 4 M_{B^{*}_{X_{n}}}
M_{(B^{*}_{Y_{n}})^{o}}}.
\end{equation}
\end{remark}

\begin{remark}
Let $(\Omega ,\nu ,\Xi ,\kappa )\in \mathcal{K}$, $X=L_{p}$ and $Y=L_{q}$, $%
1\leq q\leq 2\leq p\leq \infty $. Then using H\"{o}lder's inequality we get
\begin{equation}
M_{(B_{L_{q}}^{\ast })^{o}}=\int_{\mathbb{S}^{n-1}}\Vert \xi \Vert
_{L_{q}}^{o}d\mu \leq \int_{\mathbb{S}^{n-1}}\Vert \xi \Vert
_{L_{q^{^{\prime }}}}d\mu =M_{B_{L_{q^{^{\prime }}}}^{n}},  \label{holder}
\end{equation}%
where $1/q+1/q^{^{\prime }}=1$. Comparing Lemma~\ref{levymean} \thinspace
\thinspace %
with (\ref{ent}) and (\ref{holder}) we find
\begin{equation}
e_{n}(\Lambda U_{L_{p}},L_{q})\gg |{\lambda _{n}}|\left\{
\begin{array}{cc}
(pq^{^{\prime }})^{-1/2}, & p<\infty ,q>1, \\
(p\log n)^{-1/2}, & p<\infty ,q=1, \\
(q^{^{\prime }}\log n)^{-1/2}, & p=\infty ,q>1, \\
(\log n)^{-1}, & p=\infty ,q=1.%
\end{array}%
\right.   \label{ent-p}
\end{equation}
\end{remark}

\begin{remark}
Remind that $\tau _{N}=\mathrm{dim}\,\mathcal{T}_{N}$. Put $n=\tau _{N}$ and
$\tilde{n}=\tau _{N+1}$ in (\ref{ent-p}). Then in the case of Sobolev's
classes ${\lambda }_{n}=\theta _{N}^{-\gamma /2}$. By (\ref{weyl}) we have $%
\tau _{N}=C\theta _{N}^{d/2}(1+\epsilon _{N})$, where $\epsilon
_{N}\rightarrow 0$ as $N\rightarrow \infty $. Let
\[
\varrho _{N}:=\left\{
\begin{array}{cc}
(p/(q-1))^{-1/2}, & p<\infty ,q>1, \\
(p\log N)^{-1/2}, & p<\infty ,q=1, \\
(\log N/(q-1))^{-1/2}, & p=\infty ,q>1, \\
(\log N)^{-1}, & p=\infty ,q=1.%
\end{array}%
\right.
\]%
From (\ref{ent-p}) it follows that
\[
e_{\tilde{n}}(W_{p}^{\gamma },L_{q})\geq C\theta _{N+1}^{-\gamma /2}\varrho
_{N}=C\left( \frac{\theta _{N+1}}{\theta _{N}}\right) ^{-\gamma /2}\theta
_{N}^{-\gamma /2}\varrho _{N}\geq C\theta _{N}^{-\gamma /2}\varrho _{N},
\]%
where the last step follows from Lemma \ref{ratios}.\thinspace \thinspace
Since the sequence of entropy numbers is not increasing, then
\[
e_{m}\geq C\theta _{N}^{-\gamma /2}\varrho _{N}=C\left( \frac{\theta _{N}}{%
n^{2/d}}\right) ^{-\gamma /2}n^{-\gamma /d}\varrho _{N}\geq Cm^{-\gamma
/d}\varrho _{N},\,\,\,\forall m\in \lbrack n,\tilde{n}],
\]%
where we have used (\ref{weyl}) in the third inequality. From the last
inequality we get the following lower bounds for the entropy of Sobolev's
classes $W_{p}^{\gamma }$ in $L_{q}$,
\begin{equation}
e_{n}(W_{p}^{\gamma },L_{q})\geq Cn^{-\gamma /d}\left\{
\begin{array}{cc}
(p/(q-1))^{-1/2}, & p<\infty ,q>1, \\
(p\log n)^{-1/2}, & p<\infty ,q=1, \\
(\log n/(q-1))^{-1/2}, & p=\infty ,q>1, \\
(\log n)^{-1}, & p=\infty ,q=1.%
\end{array}%
\right.   \label{sob-ent}
\end{equation}
\end{remark}

\begin{remark}
\label{nwidlb} From \cite[Theorem 4]{kuto1} it follows that
\[
d^{[C\theta _{N}^{d/2}]}(W_{p}^{\gamma },L_{q})\geq C\theta _{N}^{-\gamma
/2},\,\,\,1<p,q<\infty .
\]%
Let $\phi \in L_{p}$, $2\leq p\leq \infty $ and $1\leq q\leq 2$. Then $\Vert
I_{\gamma }\phi \Vert _{q}\leq \Vert I_{\gamma }\phi \Vert _{2}\leq C\Vert
\phi \Vert _{2}\leq C\Vert \phi \Vert _{p}$, i.e., $I_{\gamma }\in \mathcal{L%
}(L_{p},L_{q})$ for any $\gamma >0$. It is easy to check that $I_{\gamma
}L_{p}$ is dense in $L_{q}$ since $L_{2}=\overline{\oplus _{k=0}^{\infty
}H_{k}}^{L_{2}}$ and $L_{2}$ is dense in $L_{q}$. Also, $L_{p}$ is reflexive
if $2\leq p<\infty $. Hence, for any $N\in \mathbb{N}$ and $1<p,q<\infty $,
from the duality of Kolmogorov and Gel'fand $n$-widths given by Proposition %
\ref{simple} (c),
\[
d_{[C\theta _{N+1}^{d/2}]}(W_{p}^{\gamma },L_{q})\geq C\theta
_{N+1}^{-\gamma /2}=C_{p,q}\left( \frac{\theta _{N+1}}{\theta _{N}}\right)
^{-\gamma /2}\theta _{N}^{-\gamma /2}.
\]%
By Lemma~\ref{ratios}\thinspace \thinspace\ $\lim_{N\rightarrow \infty
}\theta _{N+1}/\theta _{N}=1$. Thus, $\theta _{N+1}/\theta _{N}\leq 2$ for
some $N_{0}\in \mathbb{N}$ and any $N\geq N_{0}$. The last estimate can be
rewritten as
\[
d_{[C\theta _{N+1}^{d/2}]}(W_{p}^{\gamma },L_{q})\geq C\theta _{N}^{-\gamma
/2}.
\]%
Since the sequence of Kolmogorov's $n$-widths $d_{n}$ is non increasing,
then
\[
d_{[Cn^{d/2}]}(W_{p}^{\gamma },L_{q})\geq Cn^{-\gamma /2}
\]%
for any $n$, $\theta _{N}\leq n\leq \theta _{N+1}$. Therefore, for any $n\in
\mathbb{N}$,
\[
d_{n}(W_{p}^{\gamma },L_{q})\geq Cn^{-\gamma /d},\,\,\,1<p,q<\infty .
\]
\end{remark}

\begin{remark}
Lower bounds for Bernstein's $n$-widths may also be obtained. Let $\Lambda
=\{\lambda _{k}\},\lambda _{k}=\theta _{k}^{-\gamma /2},\gamma >0$. Define $%
\Lambda ^{-1}=\{\lambda _{k}^{-1}\}.$ Then, $\forall z = \sum_{k=0}^M
\sum_{m=1}^{d_k} c_{k,m} Y_{k,m}$ $\in \mathcal{T}_{M}$ we have
\begin{eqnarray*}
\lefteqn{\left\Vert \Lambda ^{-1} z\right\Vert _{2}^{2}} \\
& = & \left\Vert \Lambda ^{-1}\left(
\sum_{k=1}^{M}\sum_{m=1}^{d_{k}}c_{k,m}(z)Y_{m}^{k}\right) \right\Vert
_{2}^{2}=\left\Vert \sum_{k=1}^{M}\theta _{k}^{\gamma
/2}\sum_{m=1}^{d_{k}}c_{k,m}(z)Y_{m}^{k}\right\Vert _{2}^{2} \\
& = & \sum_{k=1}^{M}\theta _{k}^{\gamma }\sum_{m=1}^{d_{k}}|c_{k,m}(z)|^{2}
\; \leq \; (\max_{1\leq k\leq M}\theta _{k}^{\gamma
})\sum_{k=1}^{M}\sum_{m=1}^{d_{k}}|c_{k,m}(z)|^{2} \\
& = & \theta _{M}^{\gamma }\left\Vert z\right\Vert _{2}^{2},
\end{eqnarray*}
so that $\left\Vert \Lambda ^{-1}z\right\Vert _{2}\leq \theta _{M}^{\gamma
/2}\left\Vert z\right\Vert _{2}$. Therefore, $\theta _{M}^{-\gamma
/2}U_{2}\cap \mathcal{T}_{M}\subset W_{2}^{\gamma }$ and
\[
b_{n}(W_{2}^{\gamma },L_{q})\geq b_{n}(\theta _{M}^{-\gamma /2}U_{2}\cap
\mathcal{T}_{M},L_{q})=\theta _{M}^{-\gamma /2}b_{n}(U_{2}\cap \mathcal{T}%
_{M},L_{q}).
\]
Set $m=\mathrm{dim} \, \mathcal{T}_{M}$. By \cite[Theorem 1]{pt} there
exists a subspace $X_{s} \subset \{ \mathbb{R}^{m}, \| \cdot \|_{q^\prime}
\} $, $1 \le q \le 2$, $1/q+1/q^{^{\prime }}=1$, $\mathrm{dim} X_{s} = s >
\lambda l$, $0 < \lambda < 1$, such that
\begin{equation}  \label{ptom}
\|\alpha\|_{2}^* \leq C M_{B^*_{X_{m}}} (1-\lambda)^{-1/2}
(\|\alpha\|^{*}_{q^\prime})^o,\,\,\,\forall \alpha \in X_{s}.
\end{equation}
Let $\lambda =1/2$. Then $\left\Vert \alpha \right\Vert _{2}^* \leq
C_{1}M_{B^*_{X_{m}}}( \left\Vert \alpha \right\Vert^{*}_{q^\prime})^o $ and
by H\"{o}lder's inequality $(\left\Vert \alpha \right\Vert^{*}_{q^\prime})^o$
$\leq \left\Vert \alpha \right\Vert _{q}^*$. Hence,
\[
\left\Vert \alpha \right\Vert _{2}^* \leq C_{1}M_{B^*_{X_{m}}} \left\Vert
\alpha \right\Vert _{q}^*, \quad 1\leq q\leq 2.
\]%
Since, by Lemma~\ref{levymean}, $M_{B^*_{X_{s}}}\leq C_{2}$, $2<q^{^{\prime
}}<\infty $, we have
\[
\left\Vert \alpha \right\Vert _{2}^* \leq C_{3} \left\Vert \alpha
\right\Vert _{q}^*\,\,\,\forall \alpha \in X_{s}.
\]
Therefore $X_{s} \cap B_{q}^* \subset C_{4}X_{s} \cap B_{2}^{*}$ and since
the spaces $\mathbb{R}^{m}$ and $J\mathbb{R}^{m}=\mathcal{T}_{M}$ are
isometrically isomorphic we get $\left\Vert z\right\Vert _{2}\leq
C_{3}\left\Vert z\right\Vert _{q},$ $\forall z\in JX_{s} \subset \mathcal{T}%
_{M},s\geq \lbrack m/2].$ Hence, denoting an arbitrary $s$-dimensional
subspace of $\mathcal{T}_{M}$ by $Y_s$,
\begin{eqnarray*}
b_{s-1}(U_{2}\cap \mathcal{T}_{M},L_{q}) & = & \sup_{Y_{s} \subset
L_{q}}\sup_{\varepsilon >0}\{\varepsilon U_{q}\cap Y_s \subset U_{2}\} \\
& \geq & \sup_{\varepsilon >0}\{\varepsilon U_{q}\cap JX_s \subset U_{2}\cap
\mathcal{T}_{M}\} \\
& \geq & \sup_{\varepsilon >0}\{\varepsilon C_{3}^{-1}U_{2}\cap
JX_{s}\subset U_{2}\cap \mathcal{T}_{M}\}\geq C_{3}^{-1}.
\end{eqnarray*}
Consequently,
\[
b_{s-1}(W_{2}^{\gamma },L_{q})\geq C_{3}^{-1}\theta _{M}^{-\gamma /2},s\geq
\lbrack m/2].
\]%
Finally, applying the same line of arguments as in Remark 3.7 we get $%
b_{n}(W_{2}^{\gamma },L_{q})\geq C_{3}^{-1} n^{-\gamma /d}\,\,\,q >1$.
\end{remark}

We now turn to estimates for the upper bounds on entropy and $n$-widths.

\begin{theorem}
\label{up-width} Let $2\leq p,q\leq \infty $ and $\gamma >d/2$. Then
\[
d_{n}(W_{p}^{\gamma },L_{q})\ll n^{-\gamma /d}\left\{
\begin{array}{ll}
q^{1/2}, & 2\leq q<\infty , \\
(\log n)^{1/2}, & q=\infty .%
\end{array}%
\right.
\]
\end{theorem}

\begin{proof}
It is sufficient to consider the case $p=2$, since the case $p\geq 2$
follows by imbedding. For a given $N\in \mathbb{N}$ let $\theta _{N}$ be the
corresponding eigenvalue of the Laplace-Beltrami operator for $\mathbb{M}%
^{d} $. Put $N_{-1}:=1$, $N_{0}:=N$ and for any $k\geq 0$ let $N_{k+1}$ be
such that $\theta _{N_{k+1}-1}\leq 2^{2/\gamma }\theta _{N_{k}}\leq \theta
_{N_{k+1}}$. This is always possible to do since the sequence of eigenvalues
forms an increasing sequence with $+\infty $ as the only accumulation point.
By Lemma~\ref{ratios}, $\lim_{k\rightarrow \infty }\theta _{N_{k+1}}/\theta
_{N_{k+1}-1}=1$. Then, a simple argument shows that,
\[
\lim_{k\rightarrow \infty }\theta _{N_{k+1}}/\theta _{N_{k}}=2^{2/\gamma }.
\]%
From here we conclude that there is a $\delta (k)$, with $\delta \rightarrow
0$ as $k\rightarrow \infty $, and constants $C_{1},C_{2}>0$ such that
\begin{equation}
C_{1}(1+\delta )^{-k}2^{2k/\gamma }\theta _{N}\leq \theta _{N_{k}}\leq
C_{2}(1+\delta )^{k}2^{2k/\gamma }\theta _{N}.  \label{thetanbnd}
\end{equation}

Let $\mathcal{T}_{N_{k},N_{k+1}}:=\oplus _{l=N_{k}}^{N_{k+1}}\Xi _{l}$, and $%
\mathrm{dim}\,\mathcal{T}_{N_{k},N_{k+1}}=l_{k}$. Using (\ref{weyl}) we get
\begin{equation}
l_{k}<\mathrm{dim}\,\mathcal{T}_{N_{k+1}}\leq C\theta _{N_{k+1}}^{d/2}.
\label{dim1}
\end{equation}%
It is easy to check that
\[
I_{\gamma }(U_{2}\cap \mathcal{T}_{N_{k},N_{k+1}})\subset \theta
_{N_{k}}^{-\gamma /2}(U_{2}\cap \mathcal{T}_{N_{k},N_{k+1}}).
\]%
Thus, by Lemma~\ref{lemimbed}\thinspace \thinspace and (\ref{dim1}),
\begin{eqnarray}
U_{2}\cap \mathcal{T}_{N_{k},N_{k+1}} &\subset &Cl_{k}^{1/2-1/q}(U_{q}\cap
\mathcal{T}_{N_{k},N_{k+1}})  \notag \\
&\subset &C\theta _{N_{k+1}}^{d(1/2-1/q)/2}(U_{q}\cap \mathcal{T}%
_{N_{k},N_{k+1}}).  \label{red}
\end{eqnarray}%
Clearly, $\Vert P\Vert _{L_{2}\rightarrow L_{2}\cap \mathcal{T}%
_{N_{k},N_{k+1}}}=1$, where $P$ is the orthogonal projection. Hence, by (\ref%
{zda}),%
\[
W_{2}^{\gamma }=I_{\gamma }U_{2}\subset \bigoplus_{k=-1}^{\infty }I_{\gamma
}(U_{2}\cap \mathcal{T}_{N_{k},N_{k+1}})
\]%
\begin{equation}
\subset \bigoplus_{k=-1}^{\infty }\theta _{N_{k}}^{-\gamma /2}(U_{2}\cap
\mathcal{T}_{N_{k},N_{k+1}}).  \label{embedding}
\end{equation}

Let $\epsilon >0$ be a fixed parameter which will be specified later, $%
M:=[\epsilon ^{-1}\log (\tau _{N})]$, $m_{0}:=\tau _{N}$, $%
m_{k}:=[2^{-\epsilon k}\tau _{N}]+1$ if $1\leq k\leq M$ and $m_{k}:=0$ if $%
k>M$. Let
\[
\mu :=\sum_{k=0}^{M}m_{k}\leq \tau _{N}+\sum_{k=1}^{M}2^{-\epsilon k}\tau
_{N}+M\leq C\tau _{N}\leq C\theta _{N}^{d/2}.
\]%
Using Proposition~\ref{simple}~(a) and (b), (\ref{red}) and (\ref{embedding}%
) we find
\begin{eqnarray}
d_{\mu }(W_{2}^{\gamma },L_{q}) &\leq &C\sum_{k=0}^{M}\theta
_{N_{k}}^{-\gamma /2}d_{m_{k}}(U_{2}\cap \mathcal{T}_{N_{k},N_{k+1}},L_{q}%
\cap \mathcal{T}_{N_{k},N_{k+1}})  \notag \\
&&\hspace{1cm}+C\sum_{k=M+1}^{\infty }\theta _{N_{k}}^{-\gamma /2}\theta
_{N_{k+1}}^{d(1/2-1/q)/2}d_{0}(U_{q},L_{q})  \notag \\
:= &&\sigma _{1}+\sigma _{2},  \label{sumdef}
\end{eqnarray}%
where using the fact that $d_{0}(U_{q},L_{q})=1$,
\[
\sigma _{2}\leq C\sum_{k=M+1}^{\infty }\theta _{N_{k}}^{-\gamma /2}\theta
_{N_{k+1}}^{d(1/2-1/q)/2}.
\]

Using (\ref{thetanbnd}),
\[
\sigma _{2}\leq C\theta _{N}^{-\gamma /2+d(1/2-1/q)/2}\sum_{k\geq C\epsilon
^{-1}\log \theta _{N}}2^{-k(2/\gamma )(\gamma /2-d(1/2-1/q)/2)}(1+\delta
)^{k}.
\]%
Since $\delta >0$, and for sufficiently large $N$ we can choose $\delta $ as
small as we please, then the last series converges if $\gamma /d>1/2-1/q$.
In this case
\begin{eqnarray*}
\sigma _{2} &\leq &C\theta _{N}^{-\gamma /2+d(1/2-1/q)/2}2^{C(\log \theta
_{N})(-\gamma /2+d(1/2-1/q)/2)/\epsilon \gamma }(1+\delta )^{C(\log \theta
_{N})/\epsilon } \\
&\leq &C\theta _{N}^{-\gamma /2+d(1/2-1/q)/2}\theta _{N}^{C(-\gamma
/2+d(1/2-1/q)/2)/\epsilon \gamma }\theta _{N}^{C\eta /\epsilon },
\end{eqnarray*}%
where $\eta :=\log (1+\delta )$. Hence, if
\begin{equation}
0<\epsilon <C\gamma ^{-1}d^{-1}(1/2-1/q)^{-1}(\gamma -d(1/2-1/q)),
\label{epsbnd}
\end{equation}%
then
\begin{equation}
\sigma _{2}\leq C\theta _{N}^{-\gamma /2}.  \label{sigma2}
\end{equation}%
To complete the proof we need to get upper bounds for $\sigma _{1}$. From (%
\ref{ptom}), there exists a subspace $L_{s}^{l}\subset R_{l}^{q}=\{\mathbb{R}%
^{l},\Vert \cdot \Vert _{q}^{\ast }\}$, $\mathrm{dim}L_{s}^{l}=s>\lambda l$,
$0<\lambda <1$, such that
\[
\Vert \alpha \Vert _{2}^{\ast }\leq CM_{B_{R_{l}^{q}}^{\ast }}(1-\lambda
)^{-1/2}(\Vert \alpha \Vert _{q}^{\ast })^{o}
\]%
for any $\alpha \in L_{s}^{l}$. Put $m:=l-s$, then
\[
\Vert z\Vert _{2}\leq C(l/m)^{1/2}M_{B_{R_{l}^{q}}^{\ast }}\Vert z\Vert
_{q}^{o}
\]%
for any $z\in JL_{s}^{l}$. By duality of Kolmogorov's and Gel'fand's $n$%
-widths, recalling the definition of $m_{k}$, and letting $%
X_{m_{k}}^{l_{k}}\subset \mathcal{T}_{N_{k},N_{k+1}}$ be an arbitrary
subspace of codimension $m_{k}$, we get
\begin{eqnarray*}
\lefteqn{d_{m_{k}}(B_{2}\cap \mathcal{T}_{N_{k},N_{k+1}},L_{q}\cap \mathcal{T%
}_{N_{k},N_{k+1}})} \\
&=&d^{m_{k}}((B_{q}\cap \mathcal{T}_{N_{k},N_{k+1}})^{o},L_{2}\cap \mathcal{T%
}_{N_{k},N_{k+1}}) \\
&=&\inf_{X_{m_{k}}^{l_{k}}\subset \mathcal{T}_{N_{k},N_{k+1}}}\,\,\,\sup_{z%
\in X_{m_{k}}^{l_{k}}\cap (B_{q}\cap \mathcal{T}_{N_{k},N_{k+1}})^{o}}\Vert
z\Vert _{2} \\
&\leq &\sup_{z\in JL_{s_{k}}^{l_{k}}\cap (B_{q}\cap \mathcal{T}%
_{N_{k},N_{k+1}})^{o}}\Vert z\Vert _{2},
\end{eqnarray*}%
where $s_{k}=l_{k}-m_{k}$, since $JL_{s_{k}}^{l_{k}}$ is a specific subspace
of codimension $m_{k}$. Thus, using Lemma~\ref{ratios} and (\ref{dim1}),
\begin{eqnarray*}
\lefteqn{d_{m_{k}}(B_{2}\cap \mathcal{T}_{N_{k},N_{k+1}},L_{q}\cap \mathcal{T%
}_{N_{k},N_{k+1}})} \\
&\leq &C\left( {\frac{l_{k}}{m_{k}}}\right) ^{1/2}M_{B_{R_{l_{k}}^{q}}^{\ast
}}\sup_{z\in JL_{s_{k}}^{l_{k}}\cap (B_{q}\cap \mathcal{T}%
_{N_{k},N_{k+1}})^{o}}\Vert z\Vert _{(B_{q}\cap \mathcal{T}%
_{N_{k},N_{k+1}})^{o}} \\
&\leq &C\left( {\frac{\tau _{N_{k+1}}}{m_{k}}}\right)
^{1/2}M_{B_{R_{l_{k}}^{q}}^{\ast }} \\
&\leq &C\left( {\frac{\theta _{N_{k+1}}^{d/2}}{m_{k}}}\right)
^{1/2}M_{B_{R_{l_{k}}^{q}}^{\ast }} \\
&\leq &C\left( {\frac{((1+\delta )^{k}2^{2k/\gamma }\theta _{N})^{d/2}}{%
2^{-\epsilon k}\tau _{N}+1}}\right) ^{1/2}M_{B_{R_{l_{k}}^{q}}^{\ast }},
\end{eqnarray*}%
from (\ref{thetanbnd}). Simplifying this last expression, it follows from
Lemma~\ref{levymean} that
\begin{eqnarray*}
\lefteqn{d_{m_{k}}(B_{2}\cap \mathcal{T}_{N_{k},N_{k+1}},L_{q}\cap \mathcal{T%
}_{N_{k},N_{k+1}})} \\
&\leq &C2^{k(d/\gamma +\epsilon )/2}(1+\delta )^{kd/4}\left\{
\begin{array}{ll}
q^{1/2}, & 2\leq q<\infty , \\
(\log l_{k})^{1/2}, & q=\infty .%
\end{array}%
\right.
\end{eqnarray*}%
Let
\[
\eta _{N}:=\left\{
\begin{array}{ll}
q^{1/2}, & 2\leq q<\infty , \\
(\log \theta _{N})^{1/2}, & q=\infty .%
\end{array}%
\right.
\]%
Then, using estimate (\ref{sumdef}) and (\ref{thetanbnd}) again, we get
\begin{eqnarray*}
\sigma _{1} &\leq &C\sum_{k=0}^{M}\theta _{N_{k}}^{-\gamma
/2}d_{m_{k}}(B_{2}\cap \mathcal{T}_{N_{k},N_{k+1}},L_{q}\cap \mathcal{T}%
_{N_{k},N_{k+1}}) \\
&\leq &C\eta _{N}\sum_{k=0}^{M}\theta _{N_{k}}^{-\gamma /2}2^{k(d/\gamma
+\epsilon )/2}(1+\delta )^{kd/4} \\
&\leq &C\eta _{N}\sum_{k=0}^{\infty }(\theta _{N}2^{2k/\gamma }(1+\delta
)^{k})^{-\gamma /2}2^{k(d/\gamma +\epsilon )/2}(1+\delta )^{kd/4} \\
&\leq &C\eta _{N}\theta _{N}^{-\gamma /2}\,\sum_{k=0}^{\infty
}2^{-k(1-d/(2\gamma )-\epsilon /2)}(1+\delta )^{-k(\gamma /2-d/4)}.
\end{eqnarray*}%
The last sum is bounded for some $\delta >0$ if $\gamma >d/2$, and $%
0<\epsilon <2-d/\gamma $. Thus we must choose $\epsilon $ less than the
aforementioned and the bound given in (\ref{epsbnd}). In this case,
\begin{equation}
\sigma _{1}\leq C\theta _{N}^{-\gamma /2}\eta _{N}.  \label{sigma111}
\end{equation}%
Finally, comparing (\ref{sigma2}) and (\ref{sigma111}) we get
\[
d_{C\theta _{N}}(W_{p}^{\gamma },L_{q})\leq C\theta _{N}^{-\gamma /d}\eta
_{N},
\]%
or
\[
d_{n}(W_{p}^{\gamma },L_{q})\leq Cn^{-\gamma /d}\left\{
\begin{array}{ll}
q^{1/2}, & 2\leq q<\infty , \\
(\log n)^{1/2}, & q=\infty .%
\end{array}%
\right.
\]
\end{proof}

\begin{remark}
\label{2 leq pq} Comparing the above theorem with Remark~\ref{nwidlb}, and
applying an embedding arguments we get
\[
d_{n}(W^{\gamma}_{p}, L_{q}) \asymp n^{-\gamma/d},\,\,\,\gamma > d/2,\,\,2
\leq p < \infty,\,\,1 < q < \infty.
\]
\end{remark}

\begin{remark}
\label{compact} From Lemma~\ref{lemimbed} and (\ref{embedding}) we see that
the set $W^{\gamma}_{2}$ is bounded in $L_{q}$, $q \geq 2$ if $\gamma >
d(1/2-1/q)$. If $\gamma > d/2$ then $\lim_{n \rightarrow \infty}
d_{n}(W^{\gamma}_{2}, L_{q}) = 0$, by Theorem \ref{up-width}. Hence, $%
W^{\gamma}_{2}$ is relatively compact in $L_{q}$ (see, e.g., \cite[p. 402]%
{lgm}), or the corresponding operator $I_{\gamma}:\,L_{2} \rightarrow L_{q}$
and its conjugate $I_{\gamma}:\,L_{q^{^{\prime }}} \rightarrow L_{2}$ are
compact if $\gamma > d/2$, $1/q+1/q^{^{\prime }}=1$.
\end{remark}

We are prepared now to prove the main result of this article.

\begin{theorem}
\label{entropy} Let $\gamma >d$. Then for any $n\in \mathbb{N}$ and $1\leq
p,q\leq \infty $,
\[
e_{n}(W_{p}^{\gamma },L_{q})\geq C_{1}n^{-\gamma /d}\left\{
\begin{array}{cc}
(p/(q-1))^{-1/2}, & p<\infty ,q>1, \\
(p\log n)^{-1/2}, & p<\infty ,q=1, \\
(\log n/(q-1))^{-1/2}, & p=\infty ,q>1, \\
(\log n)^{-1}, & p=\infty ,q=1,%
\end{array}%
\right.
\]%
and
\[
e_{n}(W_{p}^{\gamma },L_{q})\leq C_{2}n^{-\gamma /d}\left\{
\begin{array}{ll}
(q/(p-1))^{1/2}, & 2\leq q<\infty ,1<p\leq 2, \\
(q\log n)^{1/2}, & 2\leq q<\infty ,p=1, \\
(\log n/(p-1))^{1/2}, & q=\infty ,1<p\leq 2, \\
\log n, & q=\infty ,p=1,%
\end{array}%
\right.
\]%
where $C_{1},C_{2}>0$. In particular, if $1<p,q<\infty $, then
\[
e_{n}(W_{p}^{\gamma },L_{q})\asymp n^{-\gamma /d}.
\]
\end{theorem}

\begin{proof}
\noindent\ From Theorem \ref{up-width}, and the duality of Kolmogorov and
Gel'fand $n$-widths, we have
\[
d^{n}(W_{q^{^{\prime }}}^{\gamma },L_{2})=d_{n}(W_{2}^{\gamma },L_{q})\ll
n^{-\gamma /d}\left\{
\begin{array}{ll}
q^{1/2}, & 2\leq q<\infty , \\
(\log n)^{1/2}, & q=\infty ,%
\end{array}%
\right.
\]%
where $1/q+1/q^{^{\prime }}=1$. Let $\{s_{n}\}$ denotes either of the
sequences $\{d_{n}\}$ or $\{d^{n}\}$. Assume that $f(l)$, $f:\mathbb{N}%
\rightarrow \mathbb{R}$ is a positive and increasing (for large $l\in
\mathbb{N}$) function such that $f(2^{j})\leq Cf(2^{j-1})$ for some fixed $C$
and any $j\in \mathbb{N}$. Then, there is a constant $C>0$ such that for all
$n\in \mathbb{N}$ we have
\[
\sup_{1\leq l\leq n}f(l)e_{l}(A,X)\leq C\sup_{1\leq l\leq
n}f(l)s_{l}(A,X),\,\,\,n\in \mathbb{N}
\]%
(see e.g. \cite{carl,et0,et1}). In particular, let $A=W_{2}^{\gamma }$, $%
X=L_{q}$,
\[
f^{\ast }(l):=l^{\gamma /d}\left\{
\begin{array}{ll}
q^{-1/2}, & 2\leq q<\infty , \\
(\log l)^{-1/2}, & q=\infty ,%
\end{array}%
\right.
\]%
then $f^{\ast }(2^{j})\leq Cf^{\ast }(2^{j-1})$ for some $C>0$ and
\[
f^{\ast }(n)e_{n}(W_{2}^{\gamma },L_{q})\leq \sup_{1\leq l\leq n}f^{\ast
}(l)d_{l}(W_{2}^{\gamma },L_{q})\leq C.
\]%
or
\begin{equation}
e_{n}(W_{2}^{\gamma },L_{q})\leq Cn^{-\gamma /d}\left\{
\begin{array}{ll}
q^{1/2}, & 2\leq q<\infty , \\
(\log n)^{1/2}, & q=\infty ,%
\end{array}%
\right.  \label{2q}
\end{equation}%
where $\gamma >d/2$ by the Theorem~\ref{up-width}. Similarly, if $\gamma
>d/2 $, then
\begin{equation}
e_{n}(W_{p}^{\gamma },L_{2})\leq Cn^{-\gamma /d}\left\{
\begin{array}{ll}
(p-1)^{-1/2}, & 1<p\leq 2, \\
(\log n)^{1/2}, & p=1.%
\end{array}%
\right.  \label{q'2}
\end{equation}%
Applying the multiplicative property of entropy numbers (see, e.g., \cite%
{pietsch}), (\ref{2q}) and (\ref{q'2}) we get,
\[
e_{n}(W_{p}^{\gamma },L_{q})=e_{n}(I_{\gamma }:L_{p}\rightarrow L_{q})
\]%
\[
=e_{n}(I_{\gamma /2}:L_{p}\rightarrow L_{2})e_{n}(I_{\gamma
/2}:L_{2}\rightarrow L_{q})
\]%
\begin{equation}
\leq Cn^{-\gamma /d}\left\{
\begin{array}{ll}
(q/(p-1))^{1/2}, & 2\leq q<\infty ,1<p\leq 2, \\
(q\log n)^{1/2}, & 2\leq q<\infty ,p=1, \\
(\log n/(p-1))^{1/2}, & q=\infty ,1<p\leq 2, \\
\log n, & q=\infty ,p=1,%
\end{array}%
\right.  \label{up}
\end{equation}%
where $\gamma /2>d/2$ or $\gamma >d$. Finally comparing (\ref{sob-ent}) and (%
\ref{up}) we get the proof.
\end{proof}

\end{document}